%% file: skeletonW3s.tex
\documentclass[12pt,reqno]{amsart}
\usepackage{amsfonts, amssymb, latexsym, graphics, color}

\usepackage{stmaryrd}		

\theoremstyle{plain}
\newtheorem{theorem}{Theorem}[section]
\newtheorem{proposition}[theorem]{Proposition}

\newtheorem{lemma}[theorem]{Lemma}

\theoremstyle{definition}
\newtheorem{definition}[theorem]{Definition}

\numberwithin{equation}{section}

\usepackage{tikz}
\usetikzlibrary{matrix,arrows}
\usetikzlibrary{positioning}
\usetikzlibrary{fit}
\usetikzlibrary{patterns}
\usetikzlibrary{arrows}

\newcommand{\id}{\mathrm{id}}

\DeclareMathOperator{\Av}{\mathrm{Av}}

\newcommand{\sepp}{\,|\,} 							

\newcommand{\dbrac}[1]{{\llbracket #1 \rrbracket}} 	

\newcommand{\pattern}[4]{										
  \raisebox{0.6ex}{
  \begin{tikzpicture}[scale=0.35, baseline=(current bounding box.center), #1]
  \useasboundingbox (0.0,-0.1) rectangle (#2+1.4,#2+1.1);
    \foreach \x/\y in {#4}
      \fill[pattern color = black!65, pattern=north east lines] (\x,\y) rectangle +(1,1);
    \draw (0.01,0.01) grid (#2+0.99,#2+0.99);
    \foreach \x/\y in {#3}
      \filldraw (\x,\y) circle (6pt);
  \end{tikzpicture}}
}

\newcommand{\patternnl}[4]{										
  \raisebox{0.6ex}{
  \begin{tikzpicture}[scale=0.35, baseline=(current bounding box.center), #1]
  \useasboundingbox (0.85,-0.1) rectangle (#2+1.4,#2+1.1);
    \foreach \x/\y in {#4}
      \fill[pattern color = black!65, pattern=north east lines] (\x,\y) rectangle +(1,1);
    \foreach \x/\y in {#3}
      \filldraw (\x,\y) circle (6pt);
  \end{tikzpicture}}
}

\newcommand{\onetwo}{\patternnl{scale=0.2}{2}{1/1,2/2}{}}		

\newcommand{\imopattern}[6]{									
  \raisebox{0.6ex}{
  \begin{tikzpicture}[scale=0.35, baseline=(current bounding box.center), #1]
  \useasboundingbox (0.0,-0.1) rectangle (#2+1.4,#2+1.1);
    \foreach \x/\y in {#6}
      \fill[pattern color = black!65, pattern=north east lines] (\x,\y) rectangle +(1,1);
    \draw (0.01,0.01) grid (#2+0.99,#2+0.99);
    \foreach \x/\y in {#4}
      \draw[fill=white] (\x,\y) circle (6pt);
    \foreach \x/\y in {#5}
      \draw[fill=white] (\x,\y) circle (11pt);
    \foreach \x/\y in {#3}
      \filldraw (\x,\y) circle (6pt);
  \end{tikzpicture}}
}

\newcommand{\patternsbm}[5]{									
  \raisebox{0.6ex}{
  \begin{tikzpicture}[scale=0.35, baseline=(current bounding box.center), #1]
  \useasboundingbox (0.0,-0.1) rectangle (#2+1.4,#2+1.1);
    \foreach \x/\y in {#4}
      \fill[pattern color = black!65, pattern=north east lines] (\x,\y) rectangle +(1,1);
    \draw (0.01,0.01) grid (#2+0.99,#2+0.99);
    \foreach \x/\y/\z/\w/\A in {#5}
       \fill[color = white!100, opacity=1, rounded corners = 1.5pt] (\x+0.125,\y+0.125) rectangle (\z-0.125,\w-0.125);
    \foreach \x/\y/\z/\w/\A in {#5}
       \draw[color = black, rounded corners = 1.5pt] (\x+0.125,\y+0.125) rectangle (\z-0.125,\w-0.125);
    \foreach \x/\y/\z/\w/\A in {#5}
       \fill[black] (\x/2+\z/2,\y/2+\w/2) node {$\scriptstyle\A$};
    \foreach \x/\y in {#3}
      \filldraw (\x,\y) circle (6pt);

  \end{tikzpicture}}
}

\newcommand{\imopatternsbm}[7]{									
  \raisebox{0.6ex}{
  \begin{tikzpicture}[scale=0.35, baseline=(current bounding box.center), #1]
  \useasboundingbox (0.0,-0.1) rectangle (#2+1.4,#2+1.1);
    \foreach \x/\y in {#6}
      \fill[pattern color = black!65, pattern=north east lines] (\x,\y) rectangle +(1,1);
    \draw (0.01,0.01) grid (#2+0.99,#2+0.99);
    \foreach \x/\y/\z/\w/\A in {#7}
       \fill[color = white!100, opacity=1, rounded corners = 1.5pt] (\x+0.125,\y+0.125) rectangle (\z-0.125,\w-0.125);
    \foreach \x/\y/\z/\w/\A in {#7}
       \draw[color = black, rounded corners = 1.5pt] (\x+0.125,\y+0.125) rectangle (\z-0.125,\w-0.125);
    \foreach \x/\y/\z/\w/\A in {#7}
       \fill[black] (\x/2+\z/2,\y/2+\w/2) node {$\scriptstyle\A$};
    \foreach \x/\y in {#4}
      \draw[fill=white] (\x,\y) circle (6pt);
    \foreach \x/\y in {#5}
      \draw[fill=white] (\x,\y) circle (11pt);
    \foreach \x/\y in {#3}
      \filldraw (\x,\y) circle (6pt);
  \end{tikzpicture}}
}

\newcommand{\decpattern}[6]{									
  \raisebox{0.6ex}{
  \begin{tikzpicture}[scale=0.35, baseline=(current bounding box.center), #1]
  \useasboundingbox (0.0,-0.1) rectangle (#2+1.4,#2+1.1);
    \foreach \x/\y in {#4}
      {
      \fill[pattern color = black!65, pattern=north east lines] (\x,\y) rectangle +(1,1);
      }
    \draw (0.01,0.01) grid (#2+0.99,#2+0.99);
       
    \foreach \x/\y/\z/\w/\A in {#5}
       {
       \fill[color = white!100, opacity=1, rounded corners=1.5pt] (\x+0.125,\y+0.125) rectangle (\z-0.125,\w-0.125);
       \draw[color = black, rounded corners=1.5pt] (\x+0.125,\y+0.125) rectangle (\z-0.125,\w-0.125);
       \fill[black] (\x/2+\z/2,\y/2+\w/2) node {$\scriptstyle\A$};
       }
    \foreach \x/\y/\z/\w/\A in {#6}
       {
       \fill[color = white!100, opacity=1, rounded corners=1.5pt] (\x+0.125,\y+0.125) rectangle (\z-0.125,\w-0.125);
       \fill[pattern color = gray, pattern=north east lines, rounded corners=1.5pt] (\x+0.125,\y+0.125) rectangle (\z-0.125,\w-0.125);
       \draw[color = black, rounded corners=1.5pt] (\x+0.125,\y+0.125) rectangle (\z-0.125,\w-0.125);
       \fill[black] (\x/2+\z/2,\y/2+\w/2) node {$\scriptstyle\A$};
       }
    \foreach \x/\y in {#3}
      \filldraw (\x,\y) circle (6pt);

  \end{tikzpicture}}
}

\newcommand{\decpatternww}[8]{								
  \raisebox{0.6ex}{
  \begin{tikzpicture}[scale=0.35, baseline=(current bounding box.center), #1]
  \useasboundingbox (0.0,-0.1) rectangle (#2+1.4,#2+1.1);
    \foreach \x/\y in {#6}
      {
      \fill[pattern color = black!65, pattern=north east lines] (\x,\y) rectangle +(1,1);
      }
    \draw (0.01,0.01) grid (#2+0.99,#2+0.99);
       
    \foreach \x/\y/\z/\w/\A in {#7}
       {
       \fill[color = white!100, opacity=1, rounded corners=1.5pt] (\x+0.125,\y+0.125) rectangle (\z-0.125,\w-0.125);
       \draw[color = black, rounded corners=1.5pt] (\x+0.125,\y+0.125) rectangle (\z-0.125,\w-0.125);
       \fill[black] (\x/2+\z/2,\y/2+\w/2) node {$\scriptstyle\A$};
       }
    \foreach \x/\y/\z/\w/\A in {#8}
       {
       \fill[color = white!100, opacity=1, rounded corners=1.5pt] (\x+0.125,\y+0.125) rectangle (\z-0.125,\w-0.125);
       \fill[pattern color = gray, pattern=north east lines, rounded corners=1.5pt] (\x+0.125,\y+0.125) rectangle (\z-0.125,\w-0.125);
       \draw[color = black, rounded corners=1.5pt] (\x+0.125,\y+0.125) rectangle (\z-0.125,\w-0.125);
       \fill[black] (\x/2+\z/2,\y/2+\w/2) node {$\scriptstyle\A$};
       }
    \foreach \x/\y in {#4}
      \draw[fill=white] (\x,\y) circle (6pt);
    \foreach \x/\y in {#5}
      \draw[fill=white] (\x,\y) circle (11pt);
    \foreach \x/\y in {#3}
      \filldraw (\x,\y) circle (6pt);

  \end{tikzpicture}}
}

\newcommand{\patternlabelled}[5]{										
  \raisebox{0.6ex}{
  \begin{tikzpicture}[scale=0.35, baseline=(current bounding box.center), #1]
    \foreach \x/\y in {#4}
      \fill[pattern color = black!65, pattern=north east lines] (\x,\y) rectangle +(1,1);
    \draw (0.01,0.01) grid (#2+0.99,#2+0.99);
    \foreach \x/\y in {#3}
      \filldraw (\x,\y) circle (6pt);
    \draw (0.5,0.5) -- (0.5,-0.5);
    \fill[black] (0.5,-1) node {$\scriptstyle(0,0)$};
    \draw (0.5,2.5) -- (0.5,4.5);
    \fill[black] (0.5,5) node {$\scriptstyle(0,2)$};
    \draw (1.5,2.5) -- (1.5,6);
    \fill[black] (1.5,6.5) node {$\scriptstyle(1,2)$};
    \draw (2.5,2.5) -- (2.5,4.5);
    \fill[black] (2.5,5) node {$\scriptstyle(2,2)$};
    
  \end{tikzpicture}}
}

\newcommand{\vinc}[3]{
\begin{tikzpicture}[baseline, inner sep = 0mm]

	\begin{scope}[yshift = 3]
	
	\foreach \x/\y in {#2}
	{
		\node at (\x*0.2,-0.14)  [label=$\y$] {};  
	}
	
	\foreach \z in {#3}
	{
		\ifnum 0<\z
			\ifnum \z<#1
				\draw[thick] (\z*0.2-0.07,-0.17) -- (\z*0.2+0.27,-0.17);
			\fi
		\fi
		
		\ifnum 0=\z
			\draw[thick] (0.08,0.1) -- (0.08,-0.17) -- (0.26,-0.17);
		\fi
		
		\ifnum \z=#1
			\draw[thick] (\z*0.2+0.13,0.1) -- (\z*0.2+0.13,-0.17) -- (\z*0.2-0.05,-0.17);
		\fi
	}
	\end{scope}
\end{tikzpicture}
}

\newcommand{\vinci}[3]{
\begin{tikzpicture}[baseline, scale = 0.8, , inner sep = 0mm]

	\begin{scope}[yshift = 3]
	
	\foreach \x/\y in {#2}
	{
		\node at (\x*0.2,-0.14)  [label=$\scriptstyle{\y}$] {};  
	}
	
	\foreach \z in {#3}
	{
		\ifnum 0<\z
			\ifnum \z<#1
				\draw[thick] (\z*0.2-0.07,-0.17) -- (\z*0.2+0.27,-0.17);
			\fi
		\fi
		
		\ifnum 0=\z
			\draw[thick] (0.08,0.1) -- (0.08,-0.17) -- (0.26,-0.17);
		\fi
		
		\ifnum \z=#1
			\draw[thick] (\z*0.2+0.13,0.1) -- (\z*0.2+0.13,-0.17) -- (\z*0.2-0.05,-0.17);
		\fi
	}
	\end{scope}
\end{tikzpicture}
}

\newcommand{\bivinc}[4]{
\begin{tikzpicture}[baseline, inner sep = 0mm]

	\begin{scope}[yshift = -2]
	
	\foreach \x/\y in {#2}
	{
		\node at (\x*0.2,0.15)   [label=$\x$] {};  
		\node at (\x*0.2,-0.14)  [label=$\y$] {};  
	}
	
	\foreach \z in {#3}
	{
		\ifnum 0<\z
			\ifnum \z<#1
				\draw[thick] (\z*0.2-0.07,-0.17) -- (\z*0.2+0.27,-0.17);
			\fi
		\fi
		
		\ifnum 0=\z
			\draw[thick] (0.08,0.1) -- (0.08,-0.17) -- (0.21,-0.17);
		\fi
		
		\ifnum \z=#1
			\draw[thick] (\z*0.2+0.13,0.1) -- (\z*0.2+0.13,-0.17) -- (\z*0.2,-0.17);
		\fi
	}
	
	\foreach \z in {#4}
	{
		\ifnum 0<\z
			\ifnum \z<#1
				\draw[thick] (\z*0.2-0.07,0.47) -- (\z*0.2+0.27,0.47);
			\fi
		\fi
		
		\ifnum 0=\z
			\draw[thick] (0.08,0.21) -- (0.08,0.47) -- (0.21,0.47);
		\fi
		
		\ifnum \z=#1
			\draw[thick] (\z*0.2+0.13,0.21) -- (\z*0.2+0.13,0.47) -- (\z*0.2,0.47);
		\fi
	}
	\end{scope}
\end{tikzpicture}
}

\newcommand{\bivincs}[5]{
\begin{tikzpicture}[baseline, inner sep = 0mm]

	\begin{scope}[yshift = -2]

	\foreach \x/\w in {#2}
	{
		\node at (\x*0.2,-0.14)   [label=$\w$] {}; 
	}
	
	\foreach \x/\w in {#3}
	{
		\node at (\x*0.2,0.15) [label=$\w$] {}; 
	}
	
	\foreach \z in {#4}
	{
		\ifnum 0<\z
			\ifnum \z<#1
				\draw[thick] (\z*0.2-0.07,-0.17) -- (\z*0.2+0.27,-0.17);
			\fi
		\fi
		
		\ifnum 0=\z
			\draw[thick] (0.08,0.1) -- (0.08,-0.17) -- (0.26,-0.17);
		\fi
		
		\ifnum \z=#1
			\draw[thick] (\z*0.2+0.13,0.1) -- (\z*0.2+0.13,-0.17) -- (\z*0.2-0.05,-0.17);
		\fi
	}
	
	\foreach \z in {#5}
	{
		\ifnum 0<\z
			\ifnum \z<#1
				\draw[thick] (\z*0.2-0.07,0.47) -- (\z*0.2+0.27,0.47);
			\fi
		\fi
		
		\ifnum 0=\z
			\draw[thick] (0.08,0.21) -- (0.08,0.47) -- (0.26,0.47);
		\fi
		
		\ifnum \z=#1
			\draw[thick] (\z*0.2+0.13,0.21) -- (\z*0.2+0.13,0.47) -- (\z*0.2-0.05,0.47);
		\fi
	}
	\end{scope}
\end{tikzpicture}
}

\newcommand{\bivinci}[4]{
\begin{tikzpicture}[baseline, scale = 0.8, inner sep = 0mm]

	\begin{scope}[yshift = -2]
	
	\foreach \x/\y in {#2}
	{
		\node at (\x*0.2,0.15)   [label=$\scriptstyle{\x}$] {};  
		\node at (\x*0.2,-0.14)  [label=$\scriptstyle{\y}$] {};  
	}
	
	\foreach \z in {#3}
	{
		\ifnum 0<\z
			\ifnum \z<#1
				\draw[thick] (\z*0.2-0.07,-0.17) -- (\z*0.2+0.27,-0.17);
			\fi
		\fi
		
		\ifnum 0=\z
			\draw[thick] (0.08,0.1) -- (0.08,-0.17) -- (0.26,-0.17);
		\fi
		
		\ifnum \z=#1
			\draw[thick] (\z*0.2+0.13,0.1) -- (\z*0.2+0.13,-0.17) -- (\z*0.2-0.05,-0.17);
		\fi
	}
	
	\foreach \z in {#4}
	{
		\ifnum 0<\z
			\ifnum \z<#1
				\draw[thick] (\z*0.2-0.07,0.47) -- (\z*0.2+0.27,0.47);
			\fi
		\fi
		
		\ifnum 0=\z
			\draw[thick] (0.08,0.21) -- (0.08,0.47) -- (0.26,0.47);
		\fi
		
		\ifnum \z=#1
			\draw[thick] (\z*0.2+0.13,0.21) -- (\z*0.2+0.13,0.47) -- (\z*0.2-0.05,0.47);
		\fi
	}
	\end{scope}
\end{tikzpicture}
}

\begin{document}

\title[Describing West-3-stack-sortable permutations]{Describing West-3-stack-sortable permutations with permutation patterns}

\author[\'Ulfarsson]{Henning \'Ulfarsson}

\address{School of Computer Science, Reykjav\'ik University, Menntavegi 1, 101 Reykjav\'ik, \mbox{Iceland}}

\email{henningu@ru.is}


\keywords{Patterns, Permutations, Sorting}

\begin{abstract}
We describe a new method for finding patterns in permutations that produce a given pattern after the permutation has been passed once through a stack. We use this method to describe West-$3$-stack-sortable permutations, that is, permutations that are sorted by three passes through a stack. We also show how the method can be applied to the bubble-sort operator. The method requires the use of mesh patterns, introduced by Br{\"a}nd{\'e}n and Claesson (2011), as well as a new type of generalized pattern we call a decorated pattern.
\end{abstract}
\maketitle

\setcounter{tocdepth}{1}
\tableofcontents

\section{Introduction}
\input{intro}

\section{Generalized permutation patterns}
\input{genpatts}

\section{Finding preimages of patterns}
\input{preim}

\section{Describing West-$3$-stack sortable permutations}
\input{desc}

\section{Open problems} \label{sec:opprob}
\input{opprob}

\bibliographystyle{amsplain}
\bibliography{skeletonW3s}

\end{document}

%% file: intro.tex
A \emph{permutation} is a one-to-one correspondence from a finite set $\{1,\dotsc,n\}$ to itself. Permutations will be written in one-line notation, so the permutation $2314$ maps $1$ to $2$, $2$ to $3$ and so on. The number of letters in a permutation will be called its \emph{length} and the set of all permutations of length $n$ is denoted $S_n$. The identity permutation $12\dotsm n$ will be denoted $\id_n$, or just $\id$ if $n$ is understood from the context, or irrelevant. For integers $a < b$ we use $\dbrac{a,b}$ to denote the set $\{a, a+1, \dotsc, b\}$.

In the 1970's Knuth~\cite{MR0378456} initiated the study of sorting and pattern avoidance in permutations. He considered the problem of sorting a permutation by passing it through a stack. A \emph{stack} is a structure that can store elements from a permutation. We can \emph{push} onto the top of the stack and elements are \emph{popped} out to the output. Consider trying to sort the permutation $231$ by a stack, as shown in Figure \ref{fig:231stack}.

\begin{figure}[htbp]
\begin{center}

\begin{tikzpicture}
 [scale = 0.225, bend angle=37.5, pre/.style={<-,shorten <=1pt,semithick},
   post/.style={->,shorten >=1pt,semithick}]
  
 \begin{scope}[xshift=54cm]
 \draw [black,thick] (-4,-0.35) -- (-1,-0.35) -- (-1,-3.35) -- (1,-3.35) -- (1,-0.35) -- (4,-0.35);
 \node[] at (0,-2.75) (stack1) {};
 
 \node[] at (2,0.35) (first) {\small 2}
 	edge [post,in=90, out=180]		(stack1);
 \node[] at (2.75,0.35) (second) {\small 3};
 \node[] at (3.5,0.35) (third) {\small 1};
 \end{scope}
 
 \begin{scope}[xshift=45cm]
 \draw [black,thick] (-4,-0.35) -- (-1,-0.35) -- (-1,-3.35) -- (1,-3.35) -- (1,-0.35) -- (4,-0.35);
 \node[] at (-2,0.45) (out1) {};
 
 \node[] at (2,0.35) (first) {\small 3};
 \node[] at (2.75,0.35) (second) {\small 1};
 
 \node[] at (0,-2.65) (stack1) {\small 2};
 \node[] at (0,-2.1) {} edge [post, in=0, out=90] (-1.7,0.25);
 \end{scope}
 
 \begin{scope}[xshift=36cm]
 \draw [black,thick] (-4,-0.35) -- (-1,-0.35) -- (-1,-3.35) -- (1,-3.35) -- (1,-0.35) -- (4,-0.35);
 \node[] at (0,-2.75) (stack1) {};
 
 \node[] at (2,0.35) (first) {\small 3}
 	edge [post,in=90, out=180]		(stack1);
 \node[] at (2.75,0.35) (second) {\small 1};
 
 \node[] at (-2,0.35) (out1) {\small 2};
 \end{scope}
 
 \begin{scope}[xshift=27cm]
 \draw [black,thick] (-4,-0.35) -- (-1,-0.35) -- (-1,-3.35) -- (1,-3.35) -- (1,-0.35) -- (4,-0.35);
 \node[] at (0,-2.25) (stack2) {};
  
 \node[] at (2,0.35) (first) {\small 1}
 	edge [post,in=90, out=180]		(stack2);
 
 \node[] at (0,-2.65) (stack1) {\small 3};
 
 \node[] at (-2,0.35) (out1) {\small 2};
 \end{scope}
 
 \begin{scope}[xshift=18cm]
 \draw [black,thick] (-4,-0.35) -- (-1,-0.35) -- (-1,-3.35) -- (1,-3.35) -- (1,-0.35) -- (4,-0.35);
 \node[] at (-2,0.25) (out1) {\small 2};
 
 \node[] at (0,-2.65) (stack1) {\small 3};
 \node[] at (0,-1.45) (stack2) {\small 1};
 \node[] at (0,-1.25) {} edge [post,in=0, out=90] (-1.7,0.25);
 \end{scope}
 
 \begin{scope}[xshift=9cm]
 \draw [black,thick] (-4,-0.35) -- (-1,-0.35) -- (-1,-3.35) -- (1,-3.35) -- (1,-0.35) -- (4,-0.35);
 \node[] at (-2,0.35) (out1) {\small 1};
 
 \node[] at (0,-2.65) (stack1) {\small 3};
 
 \node[] at (0,-2.1) {} edge [post, in=0, out=90] (-1.7,0.25);
 
 \node[] at (-2.75,0.35) (out2) {\small 2};
 \end{scope}
 
 \begin{scope}[xshift=0cm]
 \draw [black,thick] (-4,-0.35) -- (-1,-0.35) -- (-1,-3.35) -- (1,-3.35) -- (1,-0.35) -- (4,-0.35);
 
 \node[] at (-2,0.35) (out1) {\small 3};
 \node[] at (-2.75,0.35) (out2) {\small 1};
 \node[] at (-3.5,0.35) (out3) {\small 2};
 \end{scope}
  
\end{tikzpicture}

\caption{Trying, and failing, to sort $231$ with a stack. The figure is read from right to left}
\label{fig:231stack}
\end{center}
\end{figure}
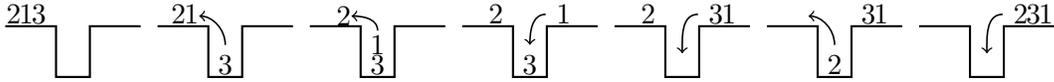

Note that we always want the elements in the stack to be increasing, from the top, since otherwise it would be impossible for the output to be sorted.  We failed to sort the permutation in one pass through the stack and therefore say that it is not \emph{stack-sortable}. Knuth~\cite{MR0378456} showed that a permutation is stack-sortable if and only if it avoids $231$ as a pattern.
We will reprove this below in Theorem~\ref{thm:Knuth}. Several variations on Knuth's original problem have been considered, such as changing the way the stack operates, and/or adding more sorting devices, see B\'ona~\cite{MR2028290} for a survey. In this paper we will consider two variations: repeatedly passing a permutation through a stack and the so-called bubble-sort operator. We introduce a new method for finding patterns in a permutation that will cause these sorting devices to output a given pattern. If the given pattern we want to be outputted is a classical pattern (defined below) we show that the mesh patterns introduced by Br\"and\'en and Claesson~\cite{BC11} suffice, but if the given pattern is itself a mesh pattern we will need to introduce a new kind of generalized pattern we call a \emph{decorated} pattern.

In the next section we review some literature on generalized permutation patterns and introduce a new generalization.

%% file: genpatts.tex
A \emph{standardization} of a list of numbers is another list of the same length such that the smallest letter in the original list has been replaced with $1$, the second smallest with $2$, and so on. The standardization of $5371$ is $3241$. A \emph{classical} (\emph{permutation}) \emph{pattern} is a permutation $p$ in $S_k$. A permutation $\pi$ in $S_n$ \emph{contains}, or \emph{has an occurrence of}, the pattern $p$ if there are indices $1 \leq i_1 < \dotsb < i_k \leq n$ such that the standardization of $\pi(i_1) \dotsm \pi(i_k)$ equals the pattern $p$. If a permutation does not contain a pattern $p$ we say that it \emph{avoids} the pattern $p$.
The permutation $\pi = 526413$ contains the pattern $p = 132$, and has three occurrences of it, given by the subsequences $264$, $263$ and $243$. We can draw the \emph{graph} of the permutation by graphing the coordinates $(i,\pi(i))$ on a grid. For example the permutation $\pi$ above is shown in Figure \ref{fig:perm526413} where we have additionally circled the occurrences of the pattern $p$.
\begin{figure}[htbp]
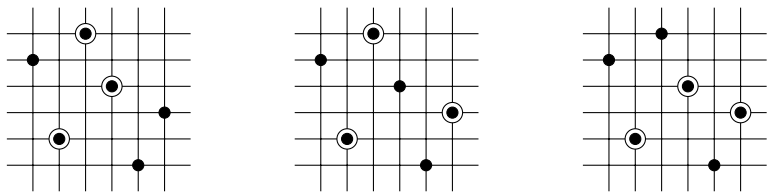

\begin{center}
\imopattern{scale=1}{ 6 }{ 1/5, 2/2, 3/6, 4/4, 5/1, 6/3 }
{}{ 2/2, 3/6, 4/4 }{}\qquad
\imopattern{scale=1}{ 6 }{ 1/5, 2/2, 3/6, 4/4, 5/1, 6/3 }
{}{ 2/2, 3/6, 6/3 }{}\qquad
\imopattern{scale=1}{ 6 }{ 1/5, 2/2, 3/6, 4/4, 5/1, 6/3 }
{}{ 2/2, 4/4, 6/3 }{}
\caption{The permutation $526413$ and three occurrences of the pattern $132$}
\label{fig:perm526413}
\end{center}
\end{figure}

The same permutation avoids the pattern $123$, since we can not find an increasing subsequence of length three in it.

Classical patterns form the base of a hierarchy of generalizations, shown in Figure \ref{fig:hierarchy}.
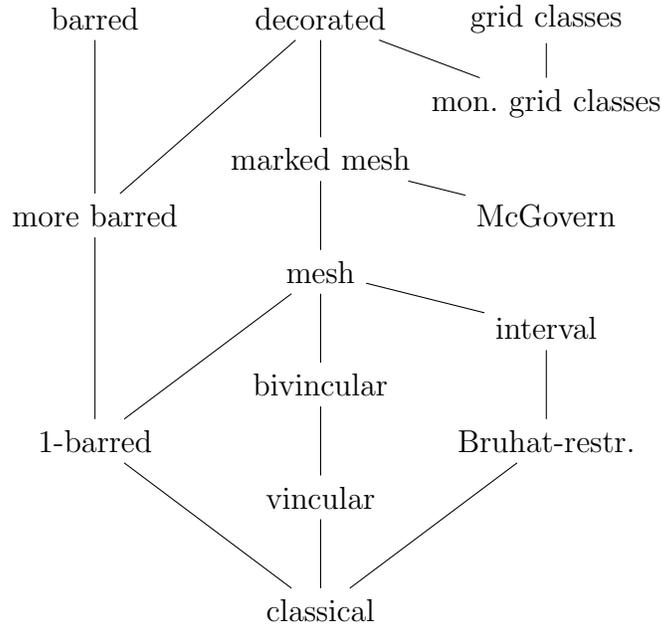
\begin{figure}[htbp]
\begin{center}
\begin{tikzpicture}
 [scale = 0.75, place/.style = {circle,draw = green!50,fill = green!20,thick,minimum size = 5pt},auto]
 
 \node[] at (0,0) (cl) {classical};
 \node[] at (0,2) (vin) {vincular};
 \node[] at (0,4) (biv) {bivincular};
 \node[] at (-4,3) (1bar) {$1$-barred};
 \node[] at (-4,7) (mbar) {more barred};
 \node[] at (-4,10.5) (bar) {barred};
 \node[] at (0,6) (m) {mesh};
 \node[] at (4,5) (inter) {interval};
 \node[] at (4,3) (Bru) {Bruhat-restr.};
 \node[] at (0,8) (mm) {marked mesh};
 \node[] at (4,7) (McG) {McGovern};
 
 \node[] at (0,10.5) (mp) {decorated};
 \node[] at (4,9) (mgrid) {mon.\ grid classes};
 \node[] at (4,10.5) (grid) {grid classes};

 \draw[-] (vin) to (cl);
 \draw[-] (biv) to (vin);
 \draw[-] (1bar) to (cl);
 \draw[-] (1bar) to (mbar);
 \draw[-] (bar) to (mbar);
 \draw[-] (m) to (1bar);
 \draw[-] (m) to (biv);
 \draw[-] (inter) to (Bru);
 \draw[-] (Bru) to (cl);
 \draw[-] (m) to (inter);
 \draw[-] (mm) to (m);
 \draw[-] (mm) to (McG);
 
 \draw[-] (mm) to (mp);
 \draw[-] (mp) to (mbar);
 \draw[-] (mp) to (mgrid);
 \draw[-] (grid) to (mgrid);
 
\end{tikzpicture}
\caption{The hierarchy of generalizations of classical patterns}
\label{fig:hierarchy}
\end{center}
\end{figure}

We will describe in detail the mesh, marked mesh and decorated patterns, as well as barred patterns, as these will be the generalizations we need. We describe the others briefly below and give references for them.

\subsection{Mesh patterns and barred patterns}
Mesh patterns were introduced in \cite{BC11}. We review them via an example.
The mesh pattern
\begin{equation}\label{eq:firstmesh}
\pattern{scale=1}{ 3 }{ 1/1, 2/3, 3/2 }{0/2,1/2,2/2} \tag{$\star$}
\end{equation}
occurs in a permutation if we can find the underlying classical pattern $132$ positioned in such a way that the shaded regions are not occupied by other entries in the permutation. Consider the permutation $526413$. From above we know that the classical pattern has three occurrences in this permutation. In Figure \ref{fig:perm526413mesh} one can see that just one of these satisfies the additional requirement that there be no additional entries in the shaded region ``between and to the left of the $3$ and the $2$''.
\begin{figure}[htbp]
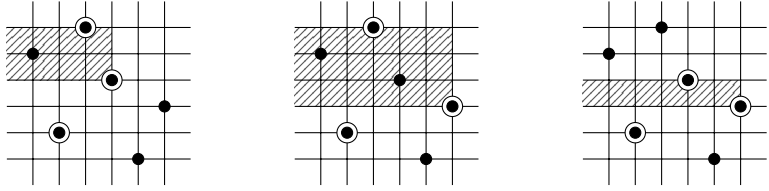

\begin{center}
\imopattern{scale=1}{ 6 }{ 1/5, 2/2, 3/6, 4/4, 5/1, 6/3 }
{}{ 2/2, 3/6, 4/4 }{0/4,0/5,1/4,1/5,2/4,2/5,3/4,3/5}\qquad
\imopattern{scale=1}{ 6 }{ 1/5, 2/2, 3/6, 4/4, 5/1, 6/3 }
{}{ 2/2, 3/6, 6/3 }{0/3,0/4,0/5,1/3,1/4,1/5,2/3,2/4,2/5,3/3,3/4,3/5,4/3,4/4,4/5,5/3,5/4,5/5}\qquad
\imopattern{scale=1}{ 6 }{ 1/5, 2/2, 3/6, 4/4, 5/1, 6/3 }
{}{ 2/2, 4/4, 6/3 }{0/3,1/3,2/3,3/3,4/3,5/3}
\caption{The permutation $526413$ and one occurrence of the mesh pattern \eqref{eq:firstmesh} above}
\label{fig:perm526413mesh}
\end{center}
\end{figure}

Another way of writing a mesh pattern is to give the underlying classical pattern, followed by the set of shaded boxes, labelled by their lower left corner (the left-most box in the bottom-most row being $(0,0)$). The mesh pattern we considered here is $(132,\{(0,2),(1,2),(2,2)\})$ and any classical pattern $p$ can be written as the mesh pattern $(p,\emptyset)$.
\begin{figure}[htbp]
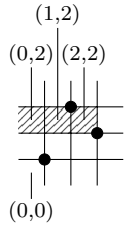

\begin{center}
\patternlabelled{scale=1}{ 3 }{ 1/1, 2/3, 3/2 }{0/2,1/2,2/2}{0/0}
\caption{The mesh pattern $(132,\{(0,2),(1,2),(2,2)\})$}
\label{fig:labelling}
\end{center}
\end{figure}

\subsubsection*{Describing the simple permutations with mesh patterns}

Recall that an \emph{interval} in a permutation is a set of entries that is consecutive in positions and values. An interval is \emph{trivial} if consists of one letter, or is the entire permutation. The elements $4653$ form a nontrivial interval in the permutation $28465317$.
\[
\imopattern{scale=1}{ 8 }{ 1/2, 2/8, 3/4, 4/6, 5/5, 6/3, 7/1, 8/7 }
{}{ 3/4, 4/6, 5/5, 6/3 }{}
\]
A permutation without nontrivial intervals is called \emph{simple}. Simple permutations have been shown to be useful in the study of classical pattern
classes, see e.g., Albert and Atkinson~\cite{MR2170110} and Brignall~\cite{MR2732823}.

Notice that the left-most, right-most, highest and lowest elements (hereafter called the \emph{boundary} elements) in the interval $4653$ form an occurrence of the mesh pattern
\[
\pattern{scale=1}{ 3 }{ 1/2, 2/3, 3/1 }{0/1,0/2,1/0,1/3,2/0,2/3,3/1,3/2}.
\]
This leads to the following proposition.
\begin{proposition}
A permutation is simple if and only if it avoids each of the patterns below, as well as their symmetries generated by the symmetries of the square.
\begin{center}
\begin{tabular}{ccc}
\pattern{scale=1}{3}{1/3, 2/1, 3/2}{0/1, 1/1, 2/0, 2/2, 2/3, 3/1} &
\pattern{scale=1}{3}{1/1, 2/2, 3/3}{3/1, 0/1, 1/0, 1/2, 1/3, 2/1} &  \\
\noalign{\smallskip}\noalign{\smallskip}\noalign{\smallskip}
\pattern{scale=1}{4}{1/4, 2/2, 3/1, 4/3}{0/1,1/1,0/2,1/2,
2/0,3/0, 4/1,4/2, 2/3,2/4,3/3,3/4} &
\pattern{scale=1}{4}{1/2, 2/1, 3/3, 4/4}{4/1,0/1,4/2,0/2,
1/0,2/0, 3/1,3/2, 1/3,1/4,2/3,2/4} &
\pattern{scale=1}{4}{1/1, 2/3, 3/2, 4/4}{0/2,1/2,0/3,1/3,
2/1,3/1, 4/2,4/3, 2/4,2/0,3/4,3/0} \\
\noalign{\smallskip}\noalign{\smallskip}\noalign{\smallskip}
\pattern{scale=1}{5}{1/5, 2/2, 3/1, 4/4, 5/3}{0/1,1/1,0/2,1/2,0/3,1/3, 2/0,3/0,4/0, 5/1,5/2,5/3, 2/4,2/5,3/4,3/5,4/4,4/5} &
\pattern{scale=1}{5}{1/2, 2/1, 3/4, 4/3, 5/5}{5/1,0/1,5/2,0/2,5/3,0/3, 1/0,2/0,3/0, 4/1,4/2,4/3, 1/4,1/5,2/4,2/5,3/4,3/5} &
\pattern{scale=1}{5}{1/5, 2/3, 3/1, 4/4, 5/2}{0/1,1/1,0/2,1/2,0/3,1/3, 2/0,3/0,4/0, 5/1,5/2,5/3, 2/4,2/5,3/4,3/5,4/4,4/5}
\end{tabular}
\end{center}
\end{proposition}

\begin{proof}
A permutation is not simple if and only if it has a non-trivial interval, and any interval must have boundary elements satisfying one of the mesh patterns above, or their symmetries. The element in the patterns that is not one of the boundary elements is added to ensure that the interval is not the entire permutation, and thus trivial.
\end{proof}

\emph{Barred} patterns were introduced by West~\cite{W90}. A barred pattern is a classical pattern with bars over some of the entries. Such a pattern is contained in a permutation if the standardization of the unbarred entries is contained in the permutation in such a way that they are not part of an occurrence of the whole barred pattern. This is best explained by considering examples. The barred pattern $3\bar{5}241$ does not occur in the permutation $416352$, since there is only one occurrence of the classical pattern we get from the unbarred entries, in the subsequence $4352$, and this
occurrence is part of an occurrence of $35241$. The permutation $5264173$ does contain this barred pattern, in the subsequence $5473$, since that is an occurrence of $3241$ that is not part of an occurrence of $35241$. In \cite{BC11} it was shown that any barred pattern with one barred entry is a mesh pattern. The barred pattern we discussed here is in fact the mesh pattern
\[
\pattern{scale=1}{ 4 }{ 1/3, 2/2, 3/4, 4/1 }{1/4}.
\]
This explains the edge from barred patterns with one bar (``$1$-barred'') to mesh patterns (``mesh'') in Figure \ref{fig:hierarchy}.

\subsection{Marked mesh patterns}
Marked mesh patterns were introduced by the author in \cite{U11} and used in joint work with Woo in \cite{UW11} in the characterization of local complete intersection Schubert varieties. They give finer control over whether a certain region in a permutation is allowed to contain elements, and if so, how many. Again we just give an example.
Consider the marked mesh pattern below.
\[
\patternsbm{scale=1}{ 3 }{ 1/1, 2/3, 3/2}{2/2}{1/0/3/2/\scriptscriptstyle{1}}
\]
The meaning of the $1$ in the region containing boxes $(1,0)$, $(1,1)$, $(2,0)$ and $(2,1)$ is that this region must contain \emph{at least} one entry. In Figure \ref{fig:perm526413marked} we see that there is exactly one occurrence of this mesh pattern in the permutation $526413$.
\begin{figure}[htbp]
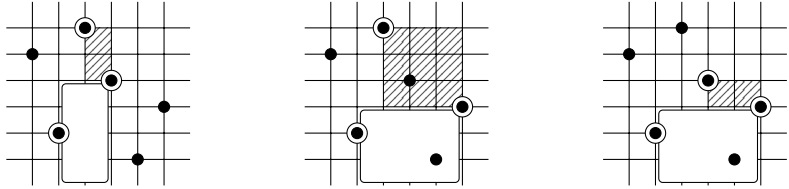

\begin{center}
\imopatternsbm{scale=1}{ 6 }{ 1/5, 2/2, 3/6, 4/4, 5/1, 6/3 }
{}{ 2/2, 3/6, 4/4 }{3/4,3/5}{2/0/4/4/{}} \qquad
\imopatternsbm{scale=1}{ 6 }{ 1/5, 2/2, 3/6, 4/4, 5/1, 6/3 }
{}{ 2/2, 3/6, 6/3 }{3/3,3/4,3/5,4/3,4/4,4/5,5/3,5/4,5/5}{2/0/6/3/{}} \qquad
\imopatternsbm{scale=1}{ 6 }{ 1/5, 2/2, 3/6, 4/4, 5/1, 6/3 }
{}{ 2/2, 4/4, 6/3 }{4/3,5/3}{2/0/6/3/{}}
\caption{The permutation $526413$ and one occurrence of a marked mesh pattern}
\label{fig:perm526413marked}
\end{center}
\end{figure}

Marked mesh patterns will be useful below when we need to add elements into an existing pattern to ensure other elements are popped out of a particular sorting device.

\subsection{Decorated patterns}
Below we will need even finer control over what is allowed inside a particular region in a pattern. We will need to control whether the entries in the region avoid a particular pattern. Consider for example the decorated pattern
\[
\decpattern{scale=1.5}{ 2 }{1/2, 2/1}{}{}{1/1/2/2/\onetwo}.
\]
The decorated region in the middle signifies that an occurrence of this pattern should be an occurrence of the underlying classical pattern $21$ that additionally does not have entries in the region that contain the pattern $12$ -- or equivalently -- whatever is in that region must be in descending order, from left to right. In Figure \ref{fig:perm526413decorated} there is an occurrence of the decorated pattern on the left and on the right we have an occurrence of the classical pattern $21$ that does not satisfy the requirements of the decorated region.
\begin{figure}[htbp]
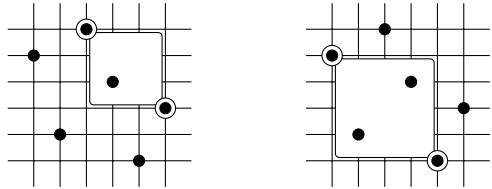

\begin{center}
\imopatternsbm{scale=1}{ 6 }{ 1/5, 2/2, 3/6, 4/4, 5/1, 6/3 }
{}{ 3/6, 6/3 }{}{3/3/6/6/{}} \qquad
\imopatternsbm{scale=1}{ 6 }{ 1/5, 2/2, 3/6, 4/4, 5/1, 6/3 }
{}{ 1/5, 5/1 }{}{1/1/5/5/{}}
\caption{The permutation $526413$ and one occurrences of a decorated pattern}
\label{fig:perm526413decorated}
\end{center}
\end{figure}

Below we state the formal definition of a decorated pattern.
\begin{definition}
A \emph{decorated pattern} $(p,\mathcal{C})$ of length $k$ consists of a classical pattern
$p$ of length $k$ and a collection $\mathcal{C}$ which contains pairs $(C, q)$
where $C$ is a subset of the square $\dbrac{0,k} \times \dbrac{0,k}$ and $q$ is some pattern,
possibly another decorated pattern. An occurrence of $(p,\mathcal{C})$ in a permutation $\pi$
is a subset $\omega$
of the diagram $G(\pi) = \{ (i,\pi(i)) \sepp 1 \leq i \leq n \}$ such that there are order-preserving
injections $\alpha, \beta: \dbrac{1,k} \to \dbrac{1,n}$ satisfying two conditions:
\begin{enumerate}
\item $\omega = \{(\alpha(i),\beta(j)) \,\colon (i,j) \in G(p)\}$.
\item Let $R_{ij} = \dbrac{\alpha(i)+1, \alpha(i+1)-1} \times \dbrac{\beta(j)+1,  \beta(j+1) -1}$,
with $\alpha(0) = 0 = \beta(0)$ and $\alpha(k+1) = n+1 = \beta(k+1)$. For each pair $(C,q)$ we let
$C' = \bigcup_{(i,j) \in C} R_{ij}$ and require that
\[
C' \cap G(\pi) \text{ avoids } q.
\]
\end{enumerate}
\end{definition}

\subsubsection*{More bars in more places} We noted above that any barred pattern with just one barred entry can be translated into a mesh pattern. If there is more than one bar, on entries that are adjacent in the pattern, this is no longer true. But we can still translate these types of barred patterns into decorated patterns. For example the barred pattern $1\bar{2}\bar{3}4$ is equivalent to the decorated pattern
\[
\decpattern{scale=1.5}{ 2 }{1/1,2/2}{}{}{1/1/2/2/\onetwo}.
\]
There are other barred patterns, such as $\bar{1}3\bar{2}4$, that can not be expressed as decorated patterns.

\subsection{Other generalizations}
Vincular patterns, sometimes called dashed patterns, were defined by Babson and Steingr\'imsson~\cite{MR1758852} and are mesh patterns where shaded regions must be vertical strips. Bivincular patterns were defined by Bousquet-M\'elou, Claesson, Dukes and Kitaev~\cite{BCDK10}, and are mesh patterns where shaded regions can be either vertical or horizontal strips.

Bruhat-restricted patterns and interval patterns were defined by Yong and Woo in \cite{MR2264071} and \cite{MR2422304}, and have applications to the study of singularities of Schubert varieties. In \cite{U11} the author showed that interval patterns are a special case of mesh patterns.
McGovern~\cite{MR2833467} showed that classical patterns, with the extra restriction that occurrences must be a union of cycles, have applications to geometry as well. It is shown
in \cite{U11} that this generalization is subsumed by marked mesh patterns.

Profile classes were introduced in Murphy and Vatter~\cite{MR2028286}, and later called
monotone grid classes in Huczynska and Vatter~\cite{MR2240760}. It is not hard to show that a permutation is a member of a particular monotone grid class if and only if it contains one pattern from a finite list of decorated patterns.

%% file: preim.tex
In this section we define a method for describing patterns that are guaranteed to produce a given pattern in a permutation after it is sorted by a stack or with the bubble-sort operator.

\subsection{The stack-sort operator} \label{subsec:stacksortop}
For a permutation $\pi$ we denote with $S(\pi)$ the image of $\pi$ after it is passed once through a stack. Recall that a permutation is stack-sortable if and only if $S(\pi) = \id$, the identity permutation. Another way to state this is that a permutation $\pi$ in $S_n$ is stack-sortable if and only if $S(\pi) \in \Av_n(21)$, where $\Av_n(21)$ is the set of permutations of length $n$ that avoid $21$. Of course $\Av_n(21) = \{\id\}$ but framing the question like this leads to a generalization: Given a pattern $p$, what conditions need to be put on $\pi$ such that $S(\pi) \in \Av_n(p)$. Equivalently, we can ask for a description of $S^{-1}(\Av_n(p))$.

Below we call permutations such that $S^k(\pi) = \id$, \emph{West-$k$-stack-sortable} permutations, since West considered this generalization from the case of one stack first.
Note that for $k > 1$ these permutations are different from the \emph{$k$-stack-sortable} permutations,
which are the permutations that can be sorted by using $k$ stacks in series without the
requirement that the permutation completely passes through one stack at a time. For example, the permutation $2341$ is not West-$2$-stack-sortable, but if we
put the entries $2,3,4$ onto the first stack, pass $1$ all the way to the end, and then use
the second stack to sort $2,3,4$ we end up with $1234$. So $2341$ is $2$-stack-sortable.
See also Open Problem \eqref{opprob:2stack} in Section \ref{sec:opprob} below.

The basic idea behind the method we are about to describe is that $S(\pi)$ has an occurrence of a pattern $p$ of length $k$ if and only if the $k$ elements making up this occurrence were present in $\pi$ as some kind of pattern before we sorted. This is of course a complete tautology but will get us pretty far. We start by showing how this idea allows us to describe stack-sortable permutations, as well as West-$2$-stack sortable permutations.

\subsubsection*{Stack-sortable permutations}
We know that $\pi$ is not sorted by the stack if and only if $S(\pi)$ contains the classical pattern $21$. Therefore consider a particular occurrence of this pattern in $S(\pi)$. Before sorting, the elements in this occurrence must have been an occurrence the pattern
\[
21 = \pattern{scale=1}{ 2 }{ 1/2, 2/1 }{}
\]
in $\pi$. In order to remain in this order the element corresponding to $2$ must be popped off the stack by a larger element before the element corresponding to $1$ enters. Thus the box $(1,2)$ must be occupied by at least one element and we have an occurrence of the marked mesh pattern
\[
\patternsbm{scale=1}{ 2 }{ 1/2, 2/1}{}{1/2/2/3/\scriptscriptstyle{1}}
\]
which is equivalent to the classical pattern
$
\pattern{scale=0.75}{ 3 }{ 1/2, 2/3, 3/1 }{}.
$
We have therefore reproven Knuth's result.
\begin{theorem}[Knuth~\cite{MR0378456}] \label{thm:Knuth}
A permutation $\pi$ is stack-sortable, i.e., $S(\pi) = \id$, if and only if it avoids the pattern $231$.
\end{theorem}

\subsubsection*{West-$2$-stack-sortable permutations}
We can similarly reprove West's result on West-$2$-stack sortable permutations, i.e., permutations $\pi$ such that $S^2(\pi) = \id$. By Knuth's result we know that $\pi$ will be sorted by two passes through the stack if and only if $S(\pi)$ avoids the pattern $231$. An occurrence of $231$ must have been either of the patterns
\begin{equation*} 
\pattern{scale=1}{ 3 }{ 1/2, 2/3, 3/1 }{}, \quad \pattern{scale=1}{ 3 }{ 1/3, 2/2, 3/1 }{}
\end{equation*}
in $\pi$.
Consider the pattern on the left. In order for the elements to stay in this order as they pass through the stack we must have an element in the box $(2,3)$ in order to pop the element corresponding to $3$ out of the stack before the smallest element enters. Now, the opposite happens for the pattern on the right. The $3$ must stay on the stack until $2$ enters, so there can be no elements in the box $(1,3)$ and then both $2$ and $3$ must leave the stack before $1$ enters. Thus the patterns above become the marked mesh patterns
\[
\patternsbm{scale=1}{ 3 }{ 1/2, 2/3, 3/1 }{}{2/3/3/4/\scriptscriptstyle{1}}, \quad \patternsbm{scale=1}{ 3 }{ 1/3, 2/2, 3/1 }{1/3}{2/3/3/4/\scriptscriptstyle{1}}.
\]
These are more naturally written as
\[
\pattern{scale=1}{ 4 }{ 1/2, 2/3, 3/4, 4/1 }{}, \quad \pattern{scale=1}{ 4 }{ 1/3, 2/2, 3/4, 4/1 }{1/3, 1/4}.
\]
The pattern on the right is equivalent to the pattern $\pattern{scale=0.75}{ 4 }{ 1/3, 2/2, 3/4, 4/1 }{1/4}$, in the sense that a permutation either contains both patterns or avoids both, by a lemma of Hilmarsson, J{\'o}nsd{\'o}ttir, Sigur{\dh}ard{\'o}ttir, Vi{\dh}arsd{\'o}ttir and {\'U}lfarsson~\cite{HJSVU11}. For these two patterns it is also easy to see directly that they are equivalent. As mentioned above this pattern is another representation of the barred pattern $3\bar{5}241$. Thus we have re-derived West's result.

\begin{theorem}[West~\cite{W90}] \label{thm:West}
A permutation $\pi$ is West-$2$-stack-sortable, i.e., $S^2(\pi) = \id$, if and only if it avoids the patterns $2341$ and $3\bar{5}241$.
\end{theorem}

Recall that an \emph{inversion} in a permutation is an occurrence of the classical pattern $21$, while a \emph{non-inversion}
is an occurrence of $12$.
We now describe how given a classical pattern $p$ the method can be used to find a set of marked mesh patterns $P$ such that
\[
S^{-1}(\Av(p)) = \Av(P).
\]
There are two things to consider when choosing what classical patterns to start from.
\begin{enumerate}

\item If two elements in $p$ are part of an inversion, they must also be part of an inversion in all patterns in $P$.

\item If two elements in $p$ are part of a non-inversion, they can either be a non-inversion or an inversion in the patterns in $P$.

\end{enumerate}

\noindent
We now consider what shadings or markings must be added.
\begin{enumerate}

\item If $a > b$ form an inversion in a particular pattern $p'$ in $P$ and are supposed to come out of the stack as an inversion, then there must be another element $c > a$ that pops the $a$ out before $b$ is pushed onto the stack, thus maintaining the inversion. If such an element is present in $p'$ we need not do anything. If there is no such element we need to mark the region above $a$ and between $a$ and $b$ with a ``$1$''.

\item If $a > b$ is an inversion in $p'$ that must become a non-inversion in $p$, then we must make sure $a$ stays on the stack until $b$ arrives, and in order to do that we must shade all the boxes above $a$ and between $a$ and $b$. In particular there can be no other elements of $p'$ in this shaded region.
If $a < b$ is a non-inversion in $p'$ then it will still be a non-inversion in $p$ after sorting.

\end{enumerate}

This suffices to generate the set of marked mesh pattens $P$ from a classical pattern $p$. We will see later that if we start with a mesh pattern $p$ then it will be more difficult to find the set $P$ and we will be forced to use the decorated patterns we introduced above.

\subsection{The bubble-sort operator}

The bubble-sort operator is a very inefficient way of sorting a permutation. It works on adjacent locations in a permutation and if the left element is larger than the right one, they are swapped. For example, one pass of bubble-sort on the permutation $52134$ produces $21345$, since the $5$ will travel from the front to the back. As another example, bubble-sorting $521634$ produces $215346$. A slight modification of the method described above works equally well for this operator. Let $B(\pi)$ denote the output of one pass of bubble-sort on $\pi$. Consider for example permutations $\pi$ such that $B(\pi) = \id$. We see that $B(\pi)$ contains the pattern $21$ if and only if $\pi$ contains
\[
\pattern{scale=1}{ 2 }{ 1/2, 2/1 }{}.
\]
To make sure that these elements stay in this order, we either need a large element in front of the $2$, which would mean the $2$ would never be moved; or we need a large element in between the $2$ and the $1$ that will stop the $2$ from moving past $1$.
We arrive at the marked mesh pattern
\[
\patternsbm{scale=1}{ 2 }{ 1/2, 2/1}{}{0/2/2/3/\scriptscriptstyle{1}}.
\]
This pattern is equivalent to the two classical patterns $231$ and $321$. Albert, Atkinson, Bouvel, Claesson and Dukes~\cite{AABCD10} first made this result explicit, that $B(\pi) = \id$ if and only if $\pi$ avoids $231$ and $321$. In the same paper the authors show that for any classical pattern $p$ with at least three left-to-right maxima, the third of which is not the final symbol of $p$, the set $B^{-1}(\Av(p))$ is not a classical pattern class, i.e., not described by classical patterns. We consider the smallest example of such a pattern, $p = 1243$

\begin{proposition}
\[
B^{-1}(\Av(1243)) = \Av \left(
\patternsbm{scale=1}{ 4 }{ 1/1, 2/2, 3/4, 4/3 }{}{0/4/4/5/\scriptscriptstyle{1}},
\patternsbm{scale=1}{ 4 }{ 1/1, 2/4, 3/2, 4/3 }{0/4,1/4,2/4}{3/4/4/5/\scriptscriptstyle{1}},
\patternsbm{scale=1}{ 4 }{ 1/2, 2/1, 3/4, 4/3 }{0/2,0/3,0/4,1/2,1/3,1/4}{2/4/4/5/\scriptscriptstyle{1}},
\patternsbm{scale=1}{ 4 }{ 1/4, 2/1, 3/2, 4/3 }{0/4,1/4,2/4}{3/4/4/5/\scriptscriptstyle{1}}
\right).
\]
\end{proposition}
\noindent
Note that all the patterns can be expanded to mesh patterns, but we would then have eight patterns instead of the four above.
\begin{proof}
It is easy to see that the underlying classical patterns are the only ones that are possible. Showing that the additional shadings and/or markings are necessary is similar for all of the patterns, so we just consider the third one. The $2$ must move past the $1$, which is achieved by adding the shading above it. The $4$ must not move past the $3$ so there must either be a large element in front of it, implying that it will never move, or between it and the $3$, implying that it will not move past the $3$.
\end{proof}

%% file: desc.tex
We now move on to the case of West-$3$-stack sortable permutations, i.e., permutations $\pi$ such that $S^3(\pi) = \id$. By West's result, stated in Theorem~\ref{thm:West}, we know that $\pi$ will be sorted by three passes through the stack if and only if $S(\pi)$ avoids the two patterns
\begin{equation*} 
W_1 = \pattern{scale=1}{ 4 }{ 1/2, 2/3, 3/4, 4/1 }{}, \quad W_2 = \pattern{scale=1}{ 4 }{ 1/3, 2/2, 3/4, 4/1 }{1/4}.
\end{equation*}
We will use the same method as we did above, but when we consider the pattern on the right, the shaded box will cause some complications and the decorated patterns introduced above will be necessary. We therefore consider the pattern on the left first.

\begin{lemma} \label{lem:I's}
Let $\pi$ be a permutation. An occurrence of $W_1$ in $S(\pi)$ comes from exactly one of the patterns below in $\pi$.
\begin{align*}
I_1 &= \pattern{scale=1}{ 5 }{ 1/2, 2/3, 3/4, 4/5, 5/1 }{}, \quad
I_2 = \pattern{scale=1}{ 5 }{ 1/2, 2/4, 3/3, 4/5, 5/1 }{2/4,2/5}, \quad
I_3 = \pattern{scale=1}{ 5 }{ 1/3, 2/2, 3/4, 4/5, 5/1 }{1/3,1/4,1/5}, \\
I_4 &= \pattern{scale=1}{ 5 }{ 1/4, 2/2, 3/3, 4/5, 5/1 }{1/4,1/5,2/4,2/5}, \quad
I_5 = \pattern{scale=1}{ 5 }{ 1/4, 2/3, 3/2, 4/5, 5/1 }{1/4,1/5,2/3,2/4,2/5}
\end{align*}
\end{lemma}

\noindent
It is easy to see that the patterns $231$, $W_1$, $I_1$ are part of a family, whose $k$-th member prohibits a permutation containing it from being West-$k$-stack sortable.

\begin{proof}
The element $5$ in all the patterns is added in to pop all the large elements out before the $1$ is pushed on the stack.
It is easy to see that the underlying classical patterns in the lemma, without the $5$, are the only possible candidates for producing an occurrence of $W_1$ in $S(\pi)$. Since proving that a particular shading must be applied to each of the patterns is similar, we only consider $I_5$. This pattern comes from the elements in the pattern $W_1$ having been arranged in the pattern
$
\pattern{scale=0.75}{ 4 }{ 1/4, 2/3, 3/2, 4/1 }{}
$
in $\pi$. In order for the elements to come out in the order we want, the $4$ must stay on the stack until both $3$ and $2$ have been pushed onto the stack. Also, $3$ must remain on the stack until $2$ arrives. We must therefore shade boxes $(1,3)$, $(1,4)$, $(2,2)$, $(2,3)$ and $(2,4)$. Finally, the elements $2$, $3$ and $4$, must leave the stack before $1$ arrives, so we must add an extra element, the $5$, in between $4$ and $1$.
\end{proof}

We now consider the pattern $W_2$, but without the shading.
\begin{lemma}
Let $\pi$ be a permutation. An occurrence of $\pattern{scale=0.75}{ 4 }{ 1/3, 2/2, 3/4, 4/1 }{}$ in $S(\pi)$ comes from exactly one of the patterns below in $\pi$.
\begin{align*}
j_1 = \patternsbm{scale=1}{ 5 }{ 1/3, 2/2, 3/4, 4/5, 5/1 }{}{1/3/2/6/\scriptscriptstyle{1}}, \quad
j_2 = \pattern{scale=1}{ 5 }{ 1/3, 2/4, 3/2, 4/5, 5/1 }{2/4,2/5}, \quad
j_3 = \patternsbm{scale=1}{ 5 }{ 1/4, 2/3, 3/2, 4/5, 5/1 }{1/4,1/5,2/4,2/5}{2/3/3/4/\scriptscriptstyle{1}}
\end{align*}
\end{lemma}

\begin{proof}
As in the proof of Lemma \ref{lem:I's} the $5$ is added in to pop out the large elements and it is easy to see that only the classical patterns underlying $j_1, j_2 and j_3$ (without the $5$) could possibly produce an occurrence of $3241$ in $S(\pi)$. 
\begin{enumerate}

\item For the classical pattern underlying $j_1$ to become $3241$ in $S(\pi)$ there must be an element in the marked region that pops the $3$ out of the stack.

\item Boxes $(2,4)$ and $(2,5)$ in the classical pattern underlying $j_2$ must be shaded to ensure that $3$ stays on the stack until $2$ is put on the stack.

\item The pattern $j_3$ comes from the elements in the pattern $W_2$ having been arranged in the pattern
$
\pattern{scale=0.75}{ 4 }{ 1/4, 2/3, 3/2, 4/1 }{}
$
in $\pi$. In order for the elements to come out in the order we want, the $4$ must stay on the stack until $2$ arrives and this explains the shaded boxes in $j_3$. Now the $3$ must be popped out before the $2$ arrives and that explains the marking in box $(2,3)$.\qedhere
\end{enumerate}
\end{proof}

We must now consider under what additional conditions the patterns $j_1$, $j_2$ and $j_3$ in the lemma will cause the correct shading in the pattern $W_2$. We express these conditions in the following lemma and two propositions.

\begin{lemma} \label{lem:j_3}
An occurrence of $j_3$ in a permutation $\pi$ will become an occurrence of $W_2$ in $S(\pi)$.
\end{lemma}

\noindent
We leave the proof to the reader. We rename the pattern $J_3$ and note that it can also be expanded into a mesh pattern.
\[
J_3 = \patternsbm{scale=1}{ 5 }{ 1/4, 2/3, 3/2, 4/5, 5/1 }{1/4,1/5,2/4,2/5}{2/3/3/4/\scriptscriptstyle{1}}
= \pattern{scale=1}{ 6 }{ 1/5, 2/3, 3/4, 4/2, 5/6, 6/1 }{1/5,1/6,2/5,2/6,3/5,3/6}
\]

\begin{proposition} \label{prop:J2}
An occurrence of $j_2$ in a permutation $\pi$ will become an occurrence of $W_2$ in $S(\pi)$ if and only
if it is part of one of the patterns below, where the elements that have been added to the pattern $j_2$ are circled.
\begin{align*}
J_{2,1} &= \imopatternsbm{scale=1}{ 6 }{ 1/3, 2/4, 3/5, 4/2, 5/6, 6/1 }{}{2/4}{1/3,1/4,1/5,1/6,2/5,2/6,3/5,3/6}{}, \quad
J_{2,2} = \decpatternww{scale=1}{ 7 }{ 1/3, 2/4, 3/6, 4/5, 5/2, 6/7, 7/1 }{}{2/4,3/6}{0/5,1/3,1/4,1/5,1/6,1/7,2/5,2/6,2/7,3/6,3/7,4/5,4/6,4/7}{}{3/5/4/6/{\onetwo}}, \quad
J_{2,3} = \decpatternww{scale=1}{ 8 }{ 1/7, 2/3, 3/4, 4/6, 5/5, 6/2, 7/8, 8/1 }{}{1/7,3/4,4/6}{1/5,1/6,1/7,1/8,2/3,2/4,2/5,2/6,2/7,2/8,3/5,3/6,3/7,3/8,4/6,4/7,4/8,5/5,5/6,5/7,5/8}{}{4/5/5/6/{\onetwo}}, \\
J_{2,4} &= \decpatternww{scale=1}{ 8 }{ 1/8, 2/3, 3/4, 4/6, 5/5, 6/2, 7/7, 8/1 }{}{1/8,3/4,4/6}{1/5,1/6,1/7,1/8,2/3,2/4,2/5,2/6,2/7,2/8,3/5,3/6,3/7,3/8,4/6,4/7,4/8,5/5,5/6,5/7,5/8}{}{4/5/5/6/{\onetwo}}, \quad
J_{2,5} = \decpatternww{scale=1}{ 7 }{ 1/3, 2/4, 3/7, 4/5, 5/2, 6/6, 7/1 }{}{2/4,3/7}{0/5,0/6,1/3,1/4,1/5,1/6,1/7,2/5,2/6,2/7,3/7,4/5,4/6,4/7}{}{3/5/4/7/{\onetwo}}, \quad
J_{2,6} = \decpatternww{scale=1}{ 8 }{ 1/8, 2/3, 3/4, 4/7, 5/5, 6/2, 7/6, 8/1 }{}{1/8,3/4,4/7}{1/5,1/6,1/7,1/8,2/3,2/4,2/5,2/6,2/7,2/8,3/5,3/6,3/7,3/8,4/7,4/8,5/5,5/6,5/7,5/8}{}{4/5/5/7/{\onetwo}}, \\
J_{2,7} &= \decpatternww{scale=1}{ 6 }{ 1/3, 2/5, 3/4, 4/2, 5/6, 6/1 }{}{2/5}{0/4,1/3,1/4,1/5,1/6,2/5,2/6,3/4,3/5,3/6}{}{2/4/3/5/{\onetwo}}, \quad
J_{2,8} = \decpatternww{scale=1}{ 7 }{ 1/6, 2/3, 3/5, 4/4, 5/2, 6/7, 7/1 }{}{1/6,3/5}{1/4,1/5,1/6,1/7,2/3,2/4,2/5,2/6,2/7,3/5,3/6,3/7,4/4,4/5,4/6,4/7}{}{3/4/4/5/{\onetwo}}, \quad
J_{2,9} = \decpatternww{scale=1}{ 7 }{ 1/7, 2/3, 3/5, 4/4, 5/2, 6/6, 7/1 }{}{1/7,3/5}{1/4,1/5,1/6,1/7,2/3,2/4,2/5,2/6,2/7,3/5,3/6,3/7,4/4,4/5,4/6,4/7}{}{3/4/4/5/{\onetwo}}, \\
J_{2,10} &= \decpatternww{scale=1}{ 6 }{ 1/3, 2/6, 3/4, 4/2, 5/5, 6/1 }{}{2/6}{0/4,0/5,1/3,1/4,1/5,1/6,2/6,3/4,3/5,3/6}{}{2/4/3/6/{\onetwo}}, \quad
J_{2,11} = \decpatternww{scale=1}{ 7 }{ 1/7, 2/3, 3/6, 4/4, 5/2, 6/5, 7/1 }{}{1/7,3/6}{1/4,1/5,1/6,1/7,2/3,2/4,2/5,2/6,2/7,3/6,3/7,4/4,4/5,4/6,4/7}{}{3/4/4/6/{\onetwo}}, \quad
J_{2,12} = \pattern{scale=1}{ 5 }{ 1/3, 2/4, 3/2, 4/5, 5/1 }{1/3,1/4,1/5,2/4,2/5}.
\end{align*}
\end{proposition}

\begin{proof}
To ensure that there are no elements in the shaded box in $W_2$ we must look at the element that pops $3$ in $j_2$. There are four different possibilities.
\begin{align*}
j_{2,1} &= \decpatternww{scale=1}{ 6 }{ 1/3, 2/4, 3/5, 4/2, 5/6, 6/1 }{}{2/4}{1/3,1/4,1/5,1/6,3/5,3/6}{}{2/5/3/7/{\onetwo}}, \quad
j_{2,2} = \decpatternww{scale=1}{ 6 }{ 1/3, 2/5, 3/4, 4/2, 5/6, 6/1 }{}{2/5}{1/3,1/4,1/5,1/6,2/5,2/6,3/4,3/5,3/6}{}{2/4/3/5/{\onetwo}}, \quad
j_{2,3} = \decpatternww{scale=1}{ 6 }{ 1/3, 2/6, 3/4, 4/2, 5/5, 6/1 }{}{2/6}{1/3,1/4,1/5,1/6,2/6,3/4,3/5,3/6}{}{2/4/3/6/{\onetwo}}, \\
j_{2,4} &= \pattern{scale=1}{ 5 }{ 1/3, 2/4, 3/2, 4/5, 5/1 }{1/3,1/4,1/5,2/4,2/5}.
\end{align*}
We explain the shadings and the decoration of the pattern $j_{2,2}$ as the others are similar. For this pattern, the size of the element that popped the $3$ from the stack was in-between the $4$ and the $5$. Since this was the element that popped the $3$ there can be no elements in boxes $(1,3), \dotsc, (1,6)$. The boxes $(2,5)$ and $(2,6)$ can not contain an element, since that would pop out the element we just added (the $5$ in $j_{2,2}$) and this element would land in the shaded box in $W_2$. Now consider the decorated box $(2,4)$. It can contain elements, but none of them are allowed to leave the stack prior to the $4$ being pushed on, since any one of them would then land in the shaded box in $W_2$. Any elements in this region must then be in descending order, or equivalently, avoid the pattern $12$.

We also need to make sure that there elements that arrived on the stack prior to $3$ are not popped out and into the shaded region in $W_2$. We do the patterns in order of difficulty, which happens to be the reverse order.
\begin{enumerate}
\item Consider the pattern $j_{2,4}$. No additional shadings are necessary, and we rename this pattern $J_{2,12}$ for future reference.
\item Consider the pattern $j_{2,3}$. If there are elements in boxes $(0,4)$ and $(0,5)$ that are still on the stack when $3$ is pushed on they will be popped by the $6$ and will land in the shaded region in $W_2$. We must therefore have these boxes empty, or an element in box $(0,6)$ that pops everything out before $3$ is pushed on. We get the two patterns
\[
J_{2,10} = \decpatternww{scale=1}{ 6 }{ 1/3, 2/6, 3/4, 4/2, 5/5, 6/1 }{}{2/6}{0/4,0/5,1/3,1/4,1/5,1/6,2/6,3/4,3/5,3/6}{}{2/4/3/6/{\onetwo}}, \quad
J_{2,11} = \decpatternww{scale=1}{ 7 }{ 1/7, 2/3, 3/6, 4/4, 5/2, 6/5, 7/1 }{}{1/7,3/6}{1/4,1/5,1/6,1/7,2/3,2/4,2/5,2/6,2/7,3/6,3/7,4/4,4/5,4/6,4/7}{}{3/4/4/6/{\onetwo}}.
\]
\item Consider the pattern $j_{2,2}$. If there are elements in box $(0,4)$ that are still on the stack when $3$ is put on they will be popped by the $5$ and will land in the shaded region $W_2$. We must therefore have this box empty, or an element in box $(0,5)$ or $(0,6)$ that pops everything out before $3$ is pushed on. We get the three patterns
\[
J_{2,7} = \decpatternww{scale=1}{ 6 }{ 1/3, 2/5, 3/4, 4/2, 5/6, 6/1 }{}{2/5}{0/4,1/3,1/4,1/5,1/6,2/5,2/6,3/4,3/5,3/6}{}{2/4/3/5/{\onetwo}}, \quad
J_{2,8} = \decpatternww{scale=1}{ 7 }{ 1/6, 2/3, 3/5, 4/4, 5/2, 6/7, 7/1 }{}{1/6,3/5}{1/4,1/5,1/6,1/7,2/3,2/4,2/5,2/6,2/7,3/5,3/6,3/7,4/4,4/5,4/6,4/7}{}{3/4/4/5/{\onetwo}}, \quad
J_{2,9} = \decpatternww{scale=1}{ 7 }{ 1/7, 2/3, 3/5, 4/4, 5/2, 6/6, 7/1 }{}{1/7,3/5}{1/4,1/5,1/6,1/7,2/3,2/4,2/5,2/6,2/7,3/5,3/6,3/7,4/4,4/5,4/6,4/7}{}{3/4/4/5/{\onetwo}}.
\]
\item Consider the pattern $j_{2,1}$. If there are elements in boxes $(0,5)$ and $(0,6)$ then they must not be popped by the elements in the decorated region (boxes $(2,6)$ and $(2,7)$). We must therefore consider three cases:
\begin{itemize}
\item The decorated region is empty: In this case we do not need to worry about anything prior to $3$ and get the pattern
\[
J_{2,1} = \imopatternsbm{scale=1}{ 6 }{ 1/3, 2/4, 3/5, 4/2, 5/6, 6/1 }{}{2/4}{1/3,1/4,1/5,1/6,2/5,2/6,3/5,3/6}{}.
\]
\item The upper box of the decorated region is empty, but not the lower box: We choose the top-most element in the lower box and add it to the pattern.
\[
\decpatternww{scale=1}{ 7 }{ 1/3, 2/4, 3/6, 4/5, 5/2, 6/7, 7/1 }{}{2/4,3/6}{1/3,1/4,1/5,1/6,1/7,2/5,2/6,2/7,3/6,3/7,4/5,4/6,4/7}{}{3/5/4/6/{\onetwo}}
\]
We can now handle this pattern the same way we handled the pattern $j_{2,2}$ above and get three patterns.
\[
J_{2,2} = \decpatternww{scale=1}{ 7 }{ 1/3, 2/4, 3/6, 4/5, 5/2, 6/7, 7/1 }{}{2/4,3/6}{0/5,1/3,1/4,1/5,1/6,1/7,2/5,2/6,2/7,3/6,3/7,4/5,4/6,4/7}{}{3/5/4/6/{\onetwo}}, \quad
J_{2,3} = \decpatternww{scale=1}{ 8 }{ 1/7, 2/3, 3/4, 4/6, 5/5, 6/2, 7/8, 8/1 }{}{1/7,3/4,4/6}{1/5,1/6,1/7,1/8,2/3,2/4,2/5,2/6,2/7,2/8,3/5,3/6,3/7,3/8,4/6,4/7,4/8,5/5,5/6,5/7,5/8}{}{4/5/5/6/{\onetwo}}, \quad
J_{2,4} = \decpatternww{scale=1}{ 8 }{ 1/8, 2/3, 3/4, 4/6, 5/5, 6/2, 7/7, 8/1 }{}{1/8,3/4,4/6}{1/5,1/6,1/7,1/8,2/3,2/4,2/5,2/6,2/7,2/8,3/5,3/6,3/7,3/8,4/6,4/7,4/8,5/5,5/6,5/7,5/8}{}{4/5/5/6/{\onetwo}}
\]
\item The upper box of the decorated region is not empty: We choose the top-most element and add it to the pattern.
\[
\decpatternww{scale=1}{ 7 }{ 1/3, 2/4, 3/7, 4/5, 5/2, 6/6, 7/1 }{}{2/4,3/7}{1/3,1/4,1/5,1/6,1/7,2/5,2/6,2/7,3/7,4/5,4/6,4/7}{}{3/5/4/7/{\onetwo}}
\]
We can again handle this pattern the same way we handled the pattern $j_{2,2}$ above and get two patterns.
\[
J_{2,5} = \decpatternww{scale=1}{ 7 }{ 1/3, 2/4, 3/7, 4/5, 5/2, 6/6, 7/1 }{}{2/4,3/7}{0/5,0/6,1/3,1/4,1/5,1/6,1/7,2/5,2/6,2/7,3/7,4/5,4/6,4/7}{}{3/5/4/7/{\onetwo}}, \quad
J_{2,6} = \decpatternww{scale=1}{ 8 }{ 1/8, 2/3, 3/4, 4/7, 5/5, 6/2, 7/6, 8/1 }{}{1/8,3/4,4/7}{1/5,1/6,1/7,1/8,2/3,2/4,2/5,2/6,2/7,2/8,3/5,3/6,3/7,3/8,4/7,4/8,5/5,5/6,5/7,5/8}{}{4/5/5/7/{\onetwo}} \qedhere
\]
\end{itemize}
\end{enumerate}
\end{proof}

We now consider the last pattern, $j_1$, and what conditions must be imposed on it in order to get $W_2$ after sorting.
\begin{proposition} \label{prop:J1}
An occurrence of $j_1$ in a permutation $\pi$ will become an occurrence of $W_2$ in $S(\pi)$ if and only
if it is part of one of the patterns below, where the elements that have been added to the pattern $j_1$ are circled.
\begin{align*}
J_{1,1} &= \imopatternsbm{scale=1}{ 6 }{ 1/3, 2/4, 3/2, 4/5, 5/6, 6/1 }{}{2/4}{1/3,1/4,1/5,1/6,2/5,2/6}{}, \quad
J_{1,2} = \decpatternww{scale=1}{ 7 }{ 1/3, 2/4, 3/6, 4/2, 5/5, 6/7, 7/1 }{}{2/4,3/6}{0/5,1/3,1/4,1/5,1/6,1/7,2/5,2/6,2/7,3/6,3/7}{}{3/5/4/6/{\onetwo}}, \quad
J_{1,3} = \decpatternww{scale=1}{ 8 }{ 1/7, 2/3, 3/4, 4/6, 5/2, 6/5, 7/8, 8/1 }{}{1/7,3/4,4/6}{1/5,1/6,1/7,1/8,2/3,2/4,2/5,2/6,2/7,2/8,3/5,3/6,3/7,3/8,4/6,4/7,4/8}{}{4/5/5/6/{\onetwo}}, \\
J_{1,4} &= \decpatternww{scale=1}{ 8 }{ 1/8, 2/3, 3/4, 4/6, 5/2, 6/5, 7/7, 8/1 }{}{1/8,3/4,4/6}{1/5,1/6,1/7,1/8,2/3,2/4,2/5,2/6,2/7,2/8,3/5,3/6,3/7,3/8,4/6,4/7,4/8}{}{4/5/5/6/{\onetwo}}, \quad
J_{1,5} = \decpatternww{scale=1}{ 7 }{ 1/3, 2/4, 3/7, 4/2, 5/5, 6/6, 7/1 }{}{2/4,3/7}{0/5,0/6,1/3,1/4,1/5,1/6,1/7,2/5,2/6,2/7,3/7}{}{3/5/4/7/{\onetwo}}, \quad
J_{1,6} = \decpatternww{scale=1}{ 8 }{ 1/8, 2/3, 3/4, 4/7, 5/2, 6/5, 7/6, 8/1 }{}{1/8,3/4,4/7}{1/5,1/6,1/7,1/8,2/3,2/4,2/5,2/6,2/7,2/8,3/5,3/6,3/7,3/8,4/7,4/8}{}{4/5/5/7/{\onetwo}}, \\
J_{1,7} &= \decpatternww{scale=1}{ 6 }{ 1/3, 2/5, 3/2, 4/4, 5/6, 6/1 }{}{2/5}{0/4,1/3,1/4,1/5,1/6,2/5,2/6}{}{2/4/3/5/{\onetwo}}, \quad
J_{1,8} = \decpatternww{scale=1}{ 7 }{ 1/6, 2/3, 3/5, 4/2, 5/4, 6/7, 7/1 }{}{1/6,3/5}{1/4,1/5,1/6,1/7,2/3,2/4,2/5,2/6,2/7,3/5,3/6,3/7}{}{3/4/4/5/{\onetwo}}, \quad
J_{1,9} = \decpatternww{scale=1}{ 7 }{ 1/7, 2/3, 3/5, 4/2, 5/4, 6/6, 7/1 }{}{1/7,3/5}{1/4,1/5,1/6,1/7,2/3,2/4,2/5,2/6,2/7,3/5,3/6,3/7}{}{3/4/4/5/{\onetwo}}, \\
J_{1,10} &= \decpatternww{scale=1}{ 6 }{ 1/3, 2/6, 3/2, 4/4, 5/5, 6/1 }{}{2/6}{0/4,0/5,1/3,1/4,1/5,1/6,2/6}{}{2/4/3/6/{\onetwo}}, \quad
J_{1,11} = \decpatternww{scale=1}{ 7 }{ 1/7, 2/3, 3/6, 4/2, 5/4, 6/5, 7/1 }{}{1/7,3/6}{1/4,1/5,1/6,1/7,2/3,2/4,2/5,2/6,2/7,3/6,3/7}{}{3/4/4/6/{\onetwo}}.
\end{align*}
\end{proposition}

\begin{proof}
To ensure that there are no elements in the shaded box in $W_2$ we must look at the element that pops $3$ in $j_1$. There are three different possibilities.
\begin{align*}
j_{1,1} &= \decpatternww{scale=1}{ 6 }{ 1/3, 2/4, 3/2, 4/5, 5/6, 6/1 }{}{2/4}{1/3,1/4,1/5,1/6}{}{2/5/3/7/{\onetwo}}, \quad
j_{1,2} = \decpatternww{scale=1}{ 6 }{ 1/3, 2/5, 3/2, 4/4, 5/6, 6/1 }{}{2/5}{1/3,1/4,1/5,1/6,2/5,2/6}{}{2/4/3/5/{\onetwo}}, \quad
j_{1,3} = \decpatternww{scale=1}{ 6 }{ 1/3, 2/6, 3/2, 4/4, 5/5, 6/1 }{}{2/6}{1/3,1/4,1/5,1/6,2/6}{}{2/4/3/6/{\onetwo}}.
\end{align*}
We also need to make sure that elements that arrived on the stack prior to $3$ are not popped out and into the shaded region in $W_2$. The derivation of the patterns in the proposition is identical to the proof of Proposition \ref{prop:J2}.
\end{proof}

Taken together, Lemmas \ref{lem:I's} and \ref{lem:j_3}, with Propositions \ref{prop:J2} and \ref{prop:J1} produce a list of $29$ patterns describing permutations that are not West-$3$-stack sortable. We can simplify this list considerably as follows:
Since the patterns $J_{1,1}, \dotsc, J_{1,6}$ all imply containment of $I_1$ we can remove them.
Similarly $J_{2,12}$ removes $J_{1,7}, \dotsc, J_{1,9}$; $I_1$ removes $J_{2,1}, \dotsc, J_{2,6}$; $I_2$ removes $J_{2,7}, \dotsc, J_{2,9}$ and $I_4$ removes $J_3$. 

We are left with the fact that a permutation is West-$3$-stack-sortable if and only if it avoids the patterns
\begin{align*}
I_1 &= \pattern{scale=1}{ 5 }{ 1/2, 2/3, 3/4, 4/5, 5/1 }{}, \quad
I_2 = \pattern{scale=1}{ 5 }{ 1/2, 2/4, 3/3, 4/5, 5/1 }{2/4,2/5}, \quad
I_3 = \pattern{scale=1}{ 5 }{ 1/3, 2/2, 3/4, 4/5, 5/1 }{1/3,1/4,1/5}, \quad
I_4 = \pattern{scale=1}{ 5 }{ 1/4, 2/2, 3/3, 4/5, 5/1 }{1/4,1/5,2/4,2/5}, \\
I_5 &= \pattern{scale=1}{ 5 }{ 1/4, 2/3, 3/2, 4/5, 5/1 }{1/4,1/5,2/3,2/4,2/5}, \quad
J_{2,12} = \pattern{scale=1}{ 5 }{ 1/3, 2/4, 3/2, 4/5, 5/1 }{1/3,1/4,1/5,2/4,2/5}, \\
J_{1,10} &= \decpatternww{scale=1}{ 6 }{ 1/3, 2/6, 3/2, 4/4, 5/5, 6/1 }{}{}{0/4,0/5,1/3,1/4,1/5,1/6,2/6}{}{2/4/3/6/{\onetwo}}, \quad
J_{2,10} = \decpatternww{scale=1}{ 6 }{ 1/3, 2/6, 3/4, 4/2, 5/5, 6/1 }{}{}{0/4,0/5,1/3,1/4,1/5,1/6,2/6,3/4,3/5,3/6}{}{2/4/3/6/{\onetwo}}, \\
J_{1,11} &= \decpatternww{scale=1}{ 7 }{ 1/7, 2/3, 3/6, 4/2, 5/4, 6/5, 7/1 }{}{}{1/4,1/5,1/6,1/7,2/3,2/4,2/5,2/6,2/7,3/6,3/7}{}{3/4/4/6/{\onetwo}}, \quad
J_{2,11} = \decpatternww{scale=1}{ 7 }{ 1/7, 2/3, 3/6, 4/4, 5/2, 6/5, 7/1 }{}{}{1/4,1/5,1/6,1/7,2/3,2/4,2/5,2/6,2/7,3/6,3/7,4/4,4/5,4/6,4/7}{}{3/4/4/6/{\onetwo}}.
\end{align*}

If the decorated region in $J_{2,10}$ contains an element in the lower box then it will produce an occurrence of $I_5$.
The same is true of $J_{2,11}$. We can therefore replace these patterns with
\[
J_{2,10} = \decpatternww{scale=1}{ 6 }{ 1/3, 2/6, 3/4, 4/2, 5/5, 6/1 }{}{}{0/4,0/5,1/3,1/4,1/5,1/6,2/4,2/6,3/4,3/5,3/6}{}{2/5/3/6/{\onetwo}}, \quad
J_{2,11} = \decpatternww{scale=1}{ 7 }{ 1/7, 2/3, 3/6, 4/4, 5/2, 6/5, 7/1 }{}{}{1/4,1/5,1/6,1/7,2/3,2/4,2/5,2/6,2/7,3/4,3/6,3/7,4/4,4/5,4/6,4/7}{}{3/5/4/6/{\onetwo}}.
\]
(Keeping the names of the patterns unchanged.)
The pattern $J_{1,10}$ can be simplified as well. The lower box in the decorated region can contain at most one element, since otherwise we would have an occurrence of $I_2$. We begin by replacing it with
\[
\decpatternww{scale=1}{ 6 }{ 1/3, 2/6, 3/2, 4/4, 5/5, 6/1 }{}{}{0/4,0/5,1/3,1/4,1/5,1/6,2/4,2/6}{}{2/5/3/6/{\onetwo}}, \quad
\decpatternww{scale=1}{ 7 }{ 1/3, 2/7, 3/5, 4/2, 5/4, 6/6, 7/1 }{}{}{0/4,0/5,0/6,1/3,1/4,1/5,1/6,1/7,2/4,2/5,2/7,3/4,3/5,3/6,3/7}{}{2/6/3/7/{\onetwo}}
\]
and notice that the second pattern above implies containment of $J_{2,10}$. We therefore replace $J_{1,10}$ with the first pattern.

The pattern $J_{1,11}$ can be simplified in the same way. The lower box in the decorated region can contain at most one element, since otherwise we would have an occurrence of $I_2$. We begin by replacing it with
\[
\decpatternww{scale=1}{ 7 }{ 1/7, 2/3, 3/6, 4/2, 5/4, 6/5, 7/1 }{}{}{1/4,1/5,1/6,1/7,2/3,2/4,2/5,2/6,2/7,3/4,3/6,3/7}{}{3/5/4/6/{\onetwo}}, \quad
\decpatternww{scale=1}{ 8 }{ 1/8, 2/3, 3/7, 4/5, 5/2, 6/4, 7/6, 8/1 }{}{}{1/4,1/5,1/6,1/7,1/8,2/3,2/4,2/5,2/6,2/7,2/8,3/4,3/5,3/7,3/8,4/4,4/5,4/6,4/7,4/8}{}{3/6/4/7/{\onetwo}}
\]
and notice that the second pattern above implies containment of $J_{2,11}$. We therefore replace $J_{1,11}$ with the first pattern.

We have therefore proven the following.
\begin{theorem} \label{thm:W3s}
A permutation $\pi$ is West-$3$-stack-sortable, i.e., $S^3(\pi) = \id$, if and only if it avoids the following decorated patterns
\begin{align*}
I_1 &= \pattern{scale=1}{ 5 }{ 1/2, 2/3, 3/4, 4/5, 5/1 }{}, \quad
I_2 = \pattern{scale=1}{ 5 }{ 1/2, 2/4, 3/3, 4/5, 5/1 }{2/4,2/5}, \quad
I_3 = \pattern{scale=1}{ 5 }{ 1/3, 2/2, 3/4, 4/5, 5/1 }{1/3,1/4,1/5}, \quad
I_4 = \pattern{scale=1}{ 5 }{ 1/4, 2/2, 3/3, 4/5, 5/1 }{1/4,1/5,2/4,2/5}, \\
I_5 &= \pattern{scale=1}{ 5 }{ 1/4, 2/3, 3/2, 4/5, 5/1 }{1/4,1/5,2/3,2/4,2/5}, \quad
J_{2,12} = \pattern{scale=1}{ 5 }{ 1/3, 2/4, 3/2, 4/5, 5/1 }{1/3,1/4,1/5,2/4,2/5}, \\
J_{1,10} &= \decpatternww{scale=1}{ 6 }{ 1/3, 2/6, 3/2, 4/4, 5/5, 6/1 }{}{}{0/4,0/5,1/3,1/4,1/5,1/6,2/4,2/6}{}{2/5/3/6/{\onetwo}}, \quad
J_{2,10} = \decpatternww{scale=1}{ 6 }{ 1/3, 2/6, 3/4, 4/2, 5/5, 6/1 }{}{}{0/4,0/5,1/3,1/4,1/5,1/6,2/4,2/6,3/4,3/5,3/6}{}{2/5/3/6/{\onetwo}}, \\
J_{1,11} &= \decpatternww{scale=1}{ 7 }{ 1/7, 2/3, 3/6, 4/2, 5/4, 6/5, 7/1 }{}{}{1/4,1/5,1/6,1/7,2/3,2/4,2/5,2/6,2/7,3/4,3/6,3/7}{}{3/5/4/6/{\onetwo}}, \quad
J_{2,11} = \decpatternww{scale=1}{ 7 }{ 1/7, 2/3, 3/6, 4/4, 5/2, 6/5, 7/1 }{}{}{1/4,1/5,1/6,1/7,2/3,2/4,2/5,2/6,2/7,3/4,3/6,3/7,4/4,4/5,4/6,4/7}{}{3/5/4/6/{\onetwo}}.
\end{align*}
\end{theorem}

%% file: opprob.tex
Above we only had use for decorated patterns where the entries in certain boxes needed to avoid a particular pattern. One might guess that
in some situations it could be useful to require the entries to contain a particular pattern instead. Furthermore, one could go as far as assigning a set $X_{(i,j)}$ to the entries in box $(i,j)$ and then placing restrictions on those sets, such as $X_{(1,1)} = \emptyset$,
$\# X_{(2,3)} \geq 2$, $X_{(0,1)}$ avoids $123$, $X_{(3,2)} \cup X_{(3,3)}$ contains the pattern $21$ etc. For a concrete example, note that a fixed point is an
occurrence of the pattern $(1, \#X_{(0,1)} = \#X_{(1,0)})$.
\newline

We end with some open problems related to the above.

\begin{enumerate}
\item The number of West-$2$-stack-sortable permutations was conjectured by West~\cite{W90} to be $2(3n)!/((n+1)!(2n+1)!)$.
This conjecture was proven by Zeilberger~\cite{Z92}. Later Dulucq, Gire and West~\cite{DGW96} found these permutations to be in bijection with rooted non-separable planar maps. The enumeration of West-$3$-stack-sortable permutations is still completely open.

\item One might hope the method we described above for finding a set of patterns $P$ such that $S^{-1}(\Av(p)) = \Av(P)$ could be turned into an algorithm. This is definitely feasible for classical patterns $p$, but might be difficult for mesh patterns $p$.

\item Our method was shown to work as well with the bubble-sort operator. Hopefully it can be applied to other similar sorting operators.

\item \label{opprob:2stack} We noted in subsection \ref{subsec:stacksortop} that there is another set of permutations called the $2$-stack-sortable permutations. It is known that classical patterns suffice to describe this set but it is not known which classical patterns. See Albert, Atkinson and Linton~\cite{MR2601797} for recent work on this.
\end{enumerate}

\noindent
The author hopes to pursue each of the open problems mentioned above in future work.

\subsection*{Acknowledgements}
The author would like to thank Anders Claesson for suggesting the problem of describing West-$3$-stack-sortable permutations, as well as for help with testing Theorem \ref{thm:W3s} with a computer. The author would also like to thank Christian Krattenthaler and an anonymous referee for detailed
comments. The author is supported by grant no.\ 090038013 from the Icelandic Research Fund.